\documentclass[leqno,a4paper]{amsart}
\usepackage{amssymb}
\usepackage{hyperref}
\usepackage{color}

\usepackage{tikz}
\usetikzlibrary{arrows.meta, positioning}
\usetikzlibrary{shapes,snakes}
\usetikzlibrary{backgrounds,calc,positioning}
\usetikzlibrary{automata,positioning}
 \newtheorem{theorem}{Theorem}[section]
 \newtheorem{corollary}[theorem]{Corollary}
 \newtheorem{lemma}[theorem]{Lemma}
 \newtheorem{proposition}[theorem]{Proposition}
 \theoremstyle{definition}
 \newtheorem{Definition}[theorem]{Definition}
 \newtheorem{remark}[theorem]{Remark}
 \newtheorem{example}[theorem]{Example}
 \numberwithin{equation}{section}

\begin{document}

\title[New characterization of $(b,c)$-inverses through polarity]{New characterization of $(b,c)$-inverses through polarity}
	\author{Btissam Laghmam \and Hassane Zguitti}	
\address{Department of Mathematics, Dhar El Mahraz Faculty of Science, Sidi Mohamed Ben Abdellah University, 30003 Fez Morocco.}
\email{btissam.laghmam@usmba.ac.ma}
\email{hassane.zguitti@usmba.ac.ma}
	
	\begin{abstract}
		In this paper we introduce the notion of $(b,c)$-polar elements in an associative ring $R$. Necessary and sufficient conditions of an element $a\in R$ to be $(b,c)$-polar are investigated. We show that an element $a\in R$ is $(b,c)$-polar if and only if $a$ is $(b,c)$-invertible. In particular the $(b,c)$-polarity is a generalization of the polarity along an element introduced by Song, Zhu and Mosi\'c \cite{polarity-along} if $b=c$, and the polarity introduced by Koliha and Patricio \cite{koliha}. Further characterizations are obtained in the Banach space context.\\


	\end{abstract}	

	\maketitle

\section{introduction}

Throughout this paper, $R$ will denote an associative ring with unity $1$. An element $a \in R$ is \textit{regular} if $a\in aRa$ i.e., $a=axa$ for some $x\in R$. Any such  $x$ is called an \textit{inner inverse} of $a$. An inner inverse of $a$ will be denoted by $a^{-}$. We denote the set of all inner invertible elements in $R$ by $R^{-}$, while the group of units in $R$ is denoted by $R^{-1} $ and the set of all left invertible (respectively right invertible) elements in $R$ by $R_{l}^{-1}$ (respectively $R_{r}^{-1}$). An element  $a$ is {\it quasinilpotent} if $1+xa\in R^{-1}$ for all $x\in comm(a)$ \cite{Harte}. Let $R^{nil}$
and $R^{qnil}$ denote, respectively, the set of all nilpotent and quasinilpotent elements in $R$. 
For any $a \in R$ we define the \textit{commutant} and \textit{double commutant} of $a$ respectively by $$comm(a)=\{x\in R \, : \, ax=xa\}  $$
$$ comm^{2}(a)=\{x\in R \, : \, xy=yx,  \text{ for all }y \in comm(a)\}.$$

Following Drazin (\cite{Drazin1}) an element $a \in R$ is said to be \textit{Drazin invertible} if there exists $x\in R$ such that $$x\in comm(a), \;  xax=x \;  \mbox{ and }a^{k+1}x=a^{k}$$ for some nonnegative integer $k$. The element $x$ is unique if it exists and is called the {\it Drazin inverse} of $a$ and is denoted by $a^{D}$. The smallest nonnegative integer $k$ satisfying the above conditions is called the Drazin index of $a$, and denoted by $ind(a)$. The set of all Drazin invertible elements in $R$ is denoted by $R^{D}$. If $ind(a)\leq 1$, then $x$ is called the {\it group inverse} of $a$. It is denoted by $a^{\sharp}$.  We denote the set of all group invertible elements in $R$ by $R^{\sharp}$.
\smallskip

Koliha and Patricio \cite{koliha} extended the notion of Drazin inverse to generalized Drazin inverse:
an element $a\in R$  is \textit{generalized Drazin invertible} if there exists $b\in R$ such that
\begin{eqnarray} \label{gdrazin}
	b \in comm^{2}(a), \quad ab^{2}=b  ,  \quad a^{2}b-a \in R^{qnil},
\end{eqnarray}

Any element $b\in R$  satisfying those conditions in (\ref{gdrazin})  is unique and is called the \textit{g-Drazin inverse} of $a$;  and it is denoted by $a^{gD}$. The set of all g-Drazin invertible elements in $R$ is denoted by $R^{gD}$. 
Koliha and Patricio gave a criterion for (generalized) Drazin invertibility by  introducing the notion of polar and quasipolar elements. An element $a\in R$ is  \textit{quasipolar} (resp. {\it polar}) if there exists an idempotent $p\in R$ such that 
 \begin{equation} \label{eq3}
	p \in comm^2(a), \;      a+p \in R^{-1} \mbox{ and } ap\in R^{qnil}\,  (\mbox{resp. }ap\in R^{nil}).
\end{equation}
 The idempotent $p$ is unique and is called the \textit{spectral idempotent} of $a$ and is denoted by $a^{\pi}$. It is proved that $a$ is generalized Drazin invertible if and only if it is quasipolar, also $a$ is Drazin invertible if and only if $a$ is polar. In this case, $a^{gD}=(a+p)^{-1}(1-p)$.
\smallskip

Based on this characterization, Wang and  Chen in \cite{WC} introduced the notion of pseudopolarity. An element $a\in R$ is said to be pseudopolar if there exists an idempotent $p\in R$ such that 
\begin{equation}\label{eqpseudop}
	p \in comm^2(a), \;      a+p \in R^{-1} \mbox{ and } ap\in R^{rad};
\end{equation}
where $R^{rad}$ denotes the Jacobson radical of $R$.
Also the idempotent $p$ is unique if it exists. They also introduced a notion between the Drazin invertibility and generalized Drazin invertibility called the pseudo Drazin invertibility: an element $a$ is pseudo Drazin invertible if there exists $b\in R$ such that 
\begin{equation}\label{pD}
	b\in comm^2(a),\, bab=b\mbox{ and }a^k-a^{k+1}b\in R^{rad}.
\end{equation}
Such element is unique if it exists and is called the pseudo Drazin inverse of $a$. Moreover, $a$ is pseudo Drazin invertible if and only if $a$ is pseudopolar, \cite{WC}.
\smallskip

Mary \cite{mary} introduced a new  generalized inverse using Green's pre-orders and relations. An element $a \in R$ will be said to be invertible along $d\in R$ if there exists $y\in R$ such that $$yad=d=day , \;  yR\subseteq dR , \;  Ry \subseteq Rd.$$ Such an $y$ is unique if it exists and called \textit{the inverse of $a$ along $d$}, denoted by $a^{\|d}$. Moreover, if $a$ is invertible along $d$ then $d$ is regular. The set of all elements invertible along $d$ in $R$ is denoted by $R^{\|d}$.

Recently, to give a new characterization of the invertibility along an element, Song, Zhu and Mosi\'c \cite{polarity-along} provided  a definition for the concept of  the polarity along an element in $R$. Let $a , d \in R$, we say that $a$ is polar along $d$ if there exists some $p\in R$  such that
 $$p=p^{2}  \in comm(da) , \;  pd=d  \text{ and } 1+da-p \in R^{-1}.$$ Which is equivalent to $$p=p^{2}  \in comm(da) , \;  pd=d  \text{ and } p\in daRda.$$
 In this case $p$ is unique and is denoted by $a^{d\pi}$. 
It is also proved that $a$ is invertible along $d$  if and only if $a$ is polar along $d$. In this case, the inverse of $a$ along $d$ is given by $a^{\|d}=(1+da-p)^{-1}d$ and $p$ is also established via $p=a^{\|d}a$. Also $a$ is invertible along $d$ if and only if $a$ is dually polar along $d$. Recall that $a$ is dually polar along $d$ if exists some $q\in R$  such that
$$q=q^{2}  \in comm(ad) , \;  dq=d  \text{ and } 1+ad-q \in R^{-1}.$$ Which is equivalent to $$q=q^{2}  \in comm(ad) , \;  dq=d  \text{ and } q\in adRad.$$ In this case $q$ is unique and is denoted by $a_{d\pi}$.
\smallskip

In 2012, Drazin introduced a class of outer inverses \cite{outer} which extended inverses along elements and others. For any $a,b,c \in R$, $a$ is said to be \textit{$(b,c)$-invertible} if  there exists $y\in R$ such that
\begin{equation}\label{bcinverse}
	y\in bR\cap Rc \; ,\;  yab=b \; , \; cay =c .
\end{equation}
If such a $y$ exists, it is unique and  called the \textit{$(b,c)$-inverse} of $a$ and denoted by $a^{\|(b,c)}$. Also if $a$ is  $(b,c)$-invertible, then $b$, $c$, and $cab$ are regular. The set of all $(b,c)$-invertible elements in $R$ is denoted by $R^{\|(b,c)}$. In the case where $b=e$ and $c=f$ such that $e$ and $f$ are idempotents, we say that $a$ is $(e,f)$-{\it Bott-Duffin invertible} if $a$ is $(e,f)$-invertible \cite{outer}. Moreover the inverse along an element is a special case of the more general class of the $(b,c)$-inverse which  occurs when $b=c$, consequently we hold   $a^{\|d}=a^{\|(d,d)}$; $a^{D}=a^{\|(a^{k},a^{k})}$, where $k$ is the index of $a$; and in particular $a^{\sharp}= a^{\|(a,a)}$.
\smallskip

In \cite{uniqness proof} Drazin also  introduced the $(b, c)$-pseudo-polarity and the $ (b, c)$-pseudo-invertibility. He shows that under some condition, the two properties are equivalent. But in general, $(b, c)$-pseudo-polarity implies $ (b, c)$-pseudo-invertibility and the converse is not true. 
\smallskip

So it is natural to ask if there exists a kind of polarity which extends polarity and polarity along an element and also characterizes the $(b,c)$-invertibility (see also \cite{uniqness proof}). More precisely, the motivation for this paper arises from the following incomplete diagram of the related concepts: \\

{\small

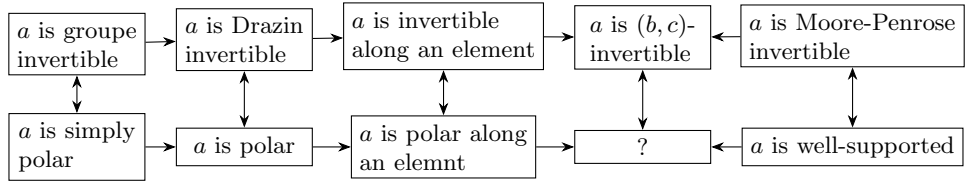
\centering

\begin{tikzpicture}[
	block/.style={
		draw,
		fill=white,
		rectangle, 
		minimum width={width("$a$ invertible ")+2pt},
		font=\small,
		>=Stealth, auto,
	},
	>=Stealth, auto	]
	\node[block, align=left](spolar){$a$ is simply \\ polar};
	\node[block,above=0.5cm of spolar, align=left](grp){$a$ is groupe \\ invertible};
	
	\node[block,right=0.4 of spolar, align=left ](polar){$a$ is polar};
	\node[block,above=0.75cm of polar, align=left](Drazin){$a$  is Drazin \\ invertible};
	\node[block,right=0.5 of polar, align=left ](along){$a$ is polar along \\ an elemnt};
	\node[block,above=0.6cm of along, align=left](along2){$a$ is invertible\\ along an element};
	\node[block,right=0.5 of along, align=left ](?){?};
	\node[block,above=0.8cm of ?, align=left](bc){$a$ is $(b,c)$-\\invertible};	
	\node[block,right=0.4 of ?, align=left ](well){$a$ is well-supported};
	\node[block,above=0.8cm of well, align=left](Moore){$a$ is Moore-Penrose \\ invertible};

	\draw[->] (grp) -- (Drazin);
	\draw[->] (Drazin) -- (along2);
	\draw[->] (along2) -- (bc);
	\draw[->] (Moore) -- (bc);
	
	\draw[<->] (grp) -- (spolar);
	\draw[<->] (polar) -- (Drazin);
	\draw[<->] (along) -- (along2);
	\draw[<->] (?) -- (bc);
	\draw[<->] (well) -- (Moore);
	
	\draw[->] (spolar) -- (polar);
	\draw[->] (polar) -- (along);
	\draw[->] (along) -- (?);
	\draw[->] (well) -- (?);

\end{tikzpicture}
}
\\

Motivated by all these notions, we introduce in this paper the notion of $(b,c)$-polarity (Definition \ref{def1}). We show that when $b=c$, then $(b,b)$-polarity coincides with polarity along $b$; which extends then the polarity along an element. Moreover, we show that an element $a$ is $(b,c)$-polar if and only if $a$ is $(b,c)$-invertible. We give then a new characterization  of $(b,c)$-invertible elements. In section 3, we introduce the concept of dually $(b,c)$-polar elements as an extension of dually polar along an element. Among other things, we show that $a$ is dually $(b,c)$-polar if and only if $a$ is $(c,b)$-invertible. The last section is devoted to illustrate $(b,c)$-polarity in the context of Banach spaces.

\section{The $(b,c)$-Polarity}
	
	We start by introducing the concept of $(b,c)$-polarity.
	\begin{Definition}  \label{def1}
		Let $a$, $b$, $c$ $\in R$, we say that $a$ is $(b,c)$-{\it polar} if there exist $ p$, $q$ $\in R $ such that
		
		\begin{enumerate}
			\item $p^{2}= p \in bRca $;
			\item $q^{2}=q \in abRc $;
			\item 	$pb=b$, $cq=c$;
			\item $cap=ca$, $qab=ab$.
		\end{enumerate}
Any idempotent $p$ ( respectively $q$) satisfying the above conditions is  called a \textit{left $(b,c)$-spectral idempotent} of $a$ (respectively  a \textit{right $(b,c)$-spectral idempotent} of $a$).
	 	\end{Definition}
In the following we show the uniqueness of the left and right $(b,c)$-spectral idempotents of a $(b,c)$-polar element. 
\begin{theorem}
	\label{thm.1}
	Let $a$, $b$, $c \in R$ such that $a$ is $(b,c)$-polar. Then $a$ has a unique left $(b,c)$-spectral idempotent and a unique right $(b,c)$-spectral idempotent.		
\end{theorem}

\begin{proof}

Suppose that $p$ and $ p^{\prime}$ are two left $(b,c)$-spectral idempotents of $a$ and $q$ and $ q^{\prime}$ are two right $(b,c)$-spectral idempotents of $a$.   \hfill \break
As $p \in bRca$, then $p=btca$ for some $t \in R$. 
It follows that $$ p-p^{\prime}p=btca-p^{\prime}btca=(b-p^{\prime}b)tca=0 \quad \text{(since  $b=pb=p^{\prime}b$) }.$$
So we obtain $$ p= p^{\prime}p.$$
Similary $p^{\prime}-pp^{\prime}=0$, and we  get $$ p^{\prime}=pp^{\prime}.$$
In other side, we have $ cap=ca=cap^{\prime}$, so  $ p-pp^{\prime}=btca-btcap^{\prime}=btca-btca=0$. 
 Hence $ p=pp^{\prime}$ and thus $p=p^{\prime}$. 
 \smallskip
 
 Similarly we show that $q=q^{\prime}$.
\end{proof}

If $a$ is $(b,c)$-polar then we denote the left $(b,c)$-spectral idempotent  $p$ by $a_{l}^{(b,c)\pi}$ and the right $(b,c)$-spectral idempotent  $q$ by $a_{r}^{(b,c)\pi}$.

\begin{example} \rm
	Let $R=\mathcal{M}_{2}(\mathbb{Z})$, and $a,b, c \in R$ such that
		
$$ a=	\left(
	\begin{array}{cc}
		0&0\\
		1&0\\
	\end{array}
	\right)  ;	
	 b= \left(
	\begin{array}{cc}
		1&-1\\
	 0&0\\
	\end{array}
	\right)  ;
	 c=\left(
	\begin{array}{cc}
		0&1\\
		0&1\\
	\end{array}
	\right).$$
	
	Then $a$ is $(b,c)$-polar with $p=a_{l}^{(b,c)\pi}= \left(
	\begin{array}{cc}
		1&0\\
		0&0\\
	\end{array}
	\right)$ and $q=a_{r}^{(b,c)\pi}= \left(
	\begin{array}{cc}
		0&0\\
		0&1\\
	\end{array}
	\right)$. 
	Indeed, a quick check, we obtain \begin{itemize}
		\item[i)] $p^{2}=p $ ; $q^{2}=q$.
		\item[ii)] $p=b\left(
		\begin{array}{cc}
			0&0\\
			-2&1\\
		\end{array}
		\right) ca $ and 
	$q=ab\left(
		\begin{array}{cc}
			1&0\\
			0&0\\
		\end{array}
		\right) c$.
		\item[iii)]  $pb=b$ ; $cq=c$.
		\item[iv)] $cap=ca$ ; $qab=ab$.
		
	\end{itemize}
		
\end{example}
Here we show that the $(b,c)$-polarity is an extension of the polarity along an element.
\begin{proposition}\label{propab}
	Let $a$ and $b\in R$. Then $a$ is $(b,b)$-polar if and only if $a$ is polar along $b$.
\end{proposition}

\begin{proof}
If $a$ is $(b,b)$-polar then we have 
$$p=p^{2}\in bRba \; ; \; q=q^{2}\in abRb \; ; \; pb=b = bq \; \text{  and } \; bap=ba=pba, $$ which implies $p=p^{2}\in comm(ba)$ and $pb=b$. Also, $p=bxba$ for some $x\in R$, then $p=bqxba \in babRbxba \subseteq baRba$.
Hence,  $a$ is polar along $b$ by \cite[ Theorem 2.4]{polarity-along}.	
\smallskip

Conversely, if $a$ is polar along $b$ then there exists a unique $p=p^2\in R$ such that $p\in comm(ba)$, $pb=b$ and $p\in baRba$. Then we have

\begin{equation}\label{eqbt1}
	\left\{\begin{array}{ccl}\,p&\in&baRba\subseteq bRba\\
		pb&=&b\\
		bap&=&pba=ba
	\end{array}
\right.
\end{equation}
On the other hand, we have $a$ is polar along $b$ if and only if $a$ is dually polar along $b$, then there exists a unique $q=q^2\in R$ such that $q\in comm(ab)$, $bq=b$ and $q\in abRab$. So 

\begin{equation}\label{eqbt2}
	\left\{\begin{array}{ccl}\,q&\in&abRab\subseteq abRb\\
		bq&=&b\\
		qab&=&abq=ab
	\end{array}
	\right.
\end{equation}
Now from (\ref{eqbt1}) and (\ref{eqbt2}), we obtain $a$ is $(b,b)$-polar.
\end{proof}

 \begin{lemma} \label{regular}
 	Let $a,b,c \in R$. If $a$ is $(b,c)$-polar, then $a$, $c $ and $cab$ are regular.
 		
 \end{lemma}

\begin{proof}

 Suppose $a$ is $(b,c)$-polar. We have  $b=a_{l}^{(b,c)\pi}b \in bRcab \subseteq bRb$, which means that $b$ is regular and admits an inner inverse denoted by $b^{-}$.
 Also $c=ca_{r}^{(b,c)\pi}\in cabRc \subseteq cRc$. An inner inverse of $c$ will be denoted by $c^{-}$. And
 $cab=ca_{r}^{(b,c)\pi}aa_{l}^{(b,c)\pi}b \in cabRcabRcab \subseteq cabRcab$. An inner inverse of $cab$ will be denoted by $(cab)^{-}$.
\end{proof}

The following theorem shows the equivalence between the $(b,c)$-polarity and the $(b,c)$-invertibility.

 \begin{theorem} \label{polarity}
 	Let  $a, b,c \in R$. Then $a$ is $(b,c)$-polar if and only if $a$ is $(b,c)$-invertible.
 	
In this case we have
\begin{itemize}
\item[i)] $ p=a_{l}^{(b,c)\pi} = a^{\parallel (b,c)}a $.
 \item[ii)] $q=a_{r}^{(b,c)\pi}= aa^{\parallel (b,c)}$ .
\item[iii)] $a^{\|(b,c)}=(1+p-bb^{-})_{l}^{-1}b(cab)^{-}c=b(cab)^{-}c(1+q-c^{-}c)_{r}^{-1}$.
With $(1+p-bb^{-})_{l}^{-1}$ is a left inverse of $1+p-bb^{-}$ and $(1+q-c^{-}c)_{r}^{-1}$ is a right inverse of $1+q-c^{-}c$.
\end{itemize}
 \end{theorem}

\begin{proof} \label{proof}

	 Suppose that $a$ is $(b,c)$-polar. To prove that $a$ is $(b,c)$-invertible, it suffices to prove that $b\in Rcab$ and $c\in cabR$ by \cite[Lemma 1]{Chen}. We have $$a_{l}^{(b,c)\pi} \in bRca,\, a_{r}^{(b,c)\pi} \in abRc,\, ca_{r}^{(b,c)\pi}=c\mbox{ and } a_{l}^{(b,c)\pi}b=b.$$ It follows that $$
	a_{l}^{(b,c)\pi}b \in bRcab \Rightarrow b\in bRcab \subseteq Rcab ,$$
	and $$ ca_{r}^{(b,c)\pi} \in cabRc \Rightarrow c\in cabRc \subseteq cabR.$$
	So $a$ is $(b,c)$-invertible.
\smallskip

	By Lemma \ref{regular} we have $b, c$ and $cab$ are regular with $b^{-}, c^{-}$ and $(cab)^{-}$  are inner inverses of $b , c$ and $cab$, respectively. Moreover we have $1+a_{l}^{(b,c)\pi}-bb^{-}$ is left invertible and $1+a_{r}^{(b,c)\pi}-c^{-}c $ is right invertible. Indeed, since $ a_{l}^{(b,c)\pi} \in bRca \subseteq bR$, $a_{l}^{(b,c)\pi}=bt $ for some $t\in R$ and we write $a_{l}^{(b,c)\pi}=bt=bb^{-}bt=bb^{-}a_{l}^{(b,c)\pi}$, so
	$$(bb^{-}+1-a_{l}^{(b,c)\pi})(1+a_{l}^{(b,c)\pi}-bb^{-})=1$$ which means that $1+a_{l}^{(b,c)\pi}-bb^{-}$ is left invertible, and we denote a left inverse of $1+a_{l}^{(b,c)\pi}-bb^{-}$ by $(1+a_{l}^{(b,c)\pi}-bb^{-})_{l}^{-1}$. Similary for $1+a_{r}^{(b,c)\pi}-c^{-}c$, as $a_{r}^{(b,c)\pi} \in abRc \subseteq Rc$ we write $a_{r}^{(b,c)\pi}=xc$ for some $x\in R$, so $a_{r}^{(b,c)\pi}=xc=xcc^{-}c=a_{r}^{(b,c)\pi}c^{-}c$ which implies $$(1+a_{r}^{(b,c)\pi}-c^{-}c)(c^{-}c+1-a_{r}^{(b,c)\pi})=1.$$ Hence $1+a_{r}^{(b,c)\pi}-c^{-}c$ is right invertible, we denote a right inverse of $1+a_{r}^{(b,c)\pi}-c^{-}c$   by $(1+a_{r}^{(b,c)\pi}-c^{-}c)_{r}^{-1}$.
\smallskip

	Now set $y=(1+a_{l}^{(b,c)\pi}-bb^{-})_{l}^{-1}b(cab)^{-}c=b(cab)^{-}c(1+a_{r}^{(b,c)\pi}-c^{-}c)_{r}^{-1}$. Then $yab=b$ ; $cay=c $ and $y\in bR\cap Rc $. Therefore, $y$ is the $(b,c)$-inverse of $a$.
\smallskip

	Conversely, suppose that $a$ is $(b,c)$-invertible  with $y$ as the $(b,c)$-inverse of $a$. Set $p=ya$ and $q=ay$, then $p$ is the left $(b,c)$-spectral idempotent of $a$ and $q$ is the right $(b,c)$-spectral idempotent of $a$. Indeed, we have
		$$
		p^{2} = yaya=ya=p \; \text{ and }
		q^{2} = ayay=ay =q .	
	$$
	
In other hand, as $y$ is the $(b,c)$-inverse of $a$ then $y \in bR\cap Rc$. Thus
	$$p =ya \in bRa \subseteq bR \; \text{   and    } p=ya \in Rca \; \text{  so,   } p=p^{2}\in bRca,  $$
	and $$ q=ay \in abR \; \text{ and  }  q=ay \in aRc \subseteq Rc \; \text{  so, } q=q^{2} \in abRc. $$
	We have  $pb= yab = b$ ; $cq= cay=c$ ;
	$cap=caya=ca$ and $qab=ayab=ab$.
	 Finally, $a$ is $(b,c)$-polar.
\end{proof}

Combining Proposition \ref{propab} and Theorem \ref{polarity} we retrieve the main result of \cite{polarity-along}.
	
\begin{corollary}
	Let $a , b \in R$. Then $a$ is polar along $b$ if and only if $a$ is invertible along $b$ if and only if $a$ is $(b,b)$-invertible if and only if $a$ is  $(b,b)$-polar.
\end{corollary}

\begin{Definition}
	Given $a, e, f\in R$, we say $a$ is $(e,f)$-{\it Bott-Duffin polar} if $a$ is $(e,f)$-polar and  $e$ and $f$ are idempotents.
\end{Definition}
  \begin{proposition}
  	Let $a,b,c \in R$. If  $a$ is $(b,c)$-polar then $a$ is $(e,f)$-Bott-Duffin polar with $e=a_{l}^{(b,c)\pi}$ and $f=a_{r}^{(b,c)\pi}.$
  	
  \end{proposition}
\begin{proof}

Set $a_{l}^{(e,f)}=a_{l}^{(b,c)}$ and $a_{r}^{(e,f)}=a_{r}^{(b,c)}$, then we obtain the result.
 \end{proof}

 \begin{corollary}
 	Let $a, e,f \in R$. Then
 	$a$ is $(e,f)$-Bott-Duffin polar if and only if $a$ is $(e,f)$-Bott-Duffin invertible.
 \end{corollary}

\section{The dual $(b,c)$-polarity}

\begin{Definition}
	 Let  $a$, $b$ and $c \in R$. We say that $a$ \textit{is dually $(b,c)$-polar} if there exist $r, s \in R$ such that
	 \begin{enumerate}
	 	\item $r^{2}=r \in acRb$ ;
	 	\item $s^{2}=s \in cRba$ ;
	 	\item $br=b$ and $sc=c$ ;
	 	\item $rac=ac$ and $bas=ba$.	
	 \end{enumerate}
\end{Definition}
Any idempotent $r$ (respectively $s$) satisfying the above conditions is  called a \textit{  dual right $(b,c)$-spectral idempotent of $a$} (respectively  a dual \textit{ left $(b,c)$-spectral idempotent of $a$}).
\begin{theorem}
	Let   $a,b ,c \in R$ such that $a$ is dually $(b,c)$-polar. Then $a$ has a unique dual right $(b,c)$-spectral idempotent and a unique dual left $(b,c)$-spectral idempotent.
\end{theorem}

\begin{proof}

 Suppose that $r$ and $r^{\prime} $ are two dual right $(b,c)$-spectral idempotent of $a$. Then we have 
 $$ r-r^{\prime}r=actb-r^{\prime}actb=actb-actb=0,$$ for some $ t\in R$.
Then $r=r^{\prime}r$. Also $$ r^{\prime}-rr^{\prime}=acxb-racxb=acxb-acxb=0,$$ for some $x\in R$. Hence $r^{\prime}=rr^{\prime}$. Consequently, we have $$ r^{\prime}r-r^{\prime}=acxbr-acxb=acx(br-b)=0,$$ so $r^{\prime}r=r^{\prime}$. Therefore we get $r^{\prime}=r$. 
\smallskip

By the same way we prove the uniqueness of the dual left $(b,c)$-spectral idempotent of $a$.
 \end{proof}

 We denote  the dual right $(b,c)$-spectral idempotent of $a$ by $ r=a_{(b,c)\pi} ^{r}$ and the dual left $(b,c)$-spectral idempotent of $a$ by $s=a_{(b,c)\pi}^{l}$.

 \begin{lemma}
 	Let $a,b,c \in R$. If $a$ is dually $(b,c)$-polar then $b,c$ and $bac$ are regular.
 \end{lemma}

\begin{proof}

 It is similar to the proof of the Lemma \ref{regular}.
 \end{proof}

\begin{theorem}
	Let  $a , b ,c \in R$. Then $a$ is dually $(b,c)$-polar if and only if $a$ is $(c,b)$-invertible.
	
In this case we have
\begin{enumerate}

  \item[i)]  $r=a_{(b,c)\pi}^{r}=aa^{\|(c,b)} $.
 \item[ii)]  $s=a_{(b,c)\pi}^{l}=a^{\|(c,b)}a$.
  \item[iii)]  $a^{\|(c,b)}=(1+a_{(b,c)\pi}^{l}-cc^{-})_{l}^{-1}c(bac)^{-}b=c(bac)^{-}b(1+a_{(b,c)\pi}^{r}-b^{-}b)_{r}^{-1}$.
\end{enumerate}
\end{theorem}

\begin{proof}

Suppose that $a$ is dually $(b,c)$-polar, then there exist $ r, s\in R$ such that 
$$ \begin{aligned}
	br&=b  \quad \text{and } \quad r\in acRb, \\
	sc&=c \quad \text{and } \quad s\in cRba.
\end{aligned}
$$
Then $br \in bacRb $, which implies that $ b \in bacRb \subseteq bacR$. Also $sc \in cRbac$, which means that $ c\in cRbac \subseteq Rbac$. Thus $a$ is $(c,b)$-invertible .
\smallskip

To obtain the formulas of the $(c,b)$-inverse of $a$, we follow the same procedure of the proof of Theorem \ref{polarity}.
\smallskip

Conversely, suppose that $a$ is $(c,b)$-invertible, then we set $r=aa^{\| (c,b)}$ and $s=a^{\| (c,b)}a$. Follow the same procedure of the proof of  Theorem  \ref{polarity}, we obtain that $a$ is dually $(b,c)$-polar with $r$ (respectively $s$) its dual right $(b,c)$-spectral idempotent (dual left $(b,c)$-spectral idempotent).
\end{proof}

It maybe that $a$ is dually $(b,c)$-polar but it is not $(b,c)$-polar as shown by the following example.

 \begin{example}\rm \label{exmple3.5}
	
	Let $R=\mathcal{M}_{2}(\mathbb{N})$, and $a,b,c \in R$ such that

	$$ a=	\left(
	\begin{array}{cc}
		1&0\\
		0&0\\
	\end{array}
	\right)  ,\,	
	 b= \left(
	\begin{array}{cc}
		0&0\\
		1&0\\
	\end{array}
	\right)  \mbox{ and }
	 c=\left(
	\begin{array}{cc}
		0&1\\
		0&0\\
	\end{array}
	\right) .$$
	
	Then $a$ is dually $(b,c)$-polar with $r= a_{(b,c)\pi}^{r}= s=a_{(b,c)\pi}^{l}=a$. Indeed, a quick check, we obtain \begin{itemize}
		\item[i)] $r^{2}=r $; 
		\item $r=ac\left(
		\begin{array}{cc}
			x_{1}&x_{2}\\
			x_{3}&1\\
		\end{array}
		\right) b $,  with $x_{1} $, $x_{2}$ and $x_{3}$ are arbitrary elements of $\mathbb{N}$;
		\item $s^{2}=s$; $s=c \left(
		\begin{array}{cc}
			t_{1}&t_{2}\\
			t_{3}&1\\
		\end{array}
		\right) ba$, with $t_{1}$, $t_{2}$ and $t_{3}$ are arbitrary elements of $\mathbb{N}$;
		\item $br=b$ and  $sc=c$;
		\item $rac=ac$ and  $bas=ba$.
		
	\end{itemize}	
	Notice that in this example, $a$ is not $(b,c)$-polar because $a$ is not $(b,c)$-invertible since $ab=0$.
\end{example}

\begin{corollary}
	\label{corol1}
	Let  $a,b,c \in R$. Then  $a$ is $(b,c)$-polar if and only if $a$ is dually $(c,b)$-polar.

\end{corollary}

 \begin{proof}

We have  $a$ is $(b,c)$-polar $\Longleftrightarrow a$ is $(b,c)$-invertible and
  $ a $ is dually $(c,b)$-polar $\Longleftrightarrow a $ is $(b,c)$-invertible.
 So we get that $a$ is $(b,c)$-polar if and only if $a$ is dually $(c,b)$-polar.

 \end{proof}

\begin{theorem} \label{thm2}
	Let  $a, b, c \in R$. If $a$ is both $(b,c)$ and $(c,b)$-invertible such that $aba\in comm(c)$ and $aca\in comm(b)$, then  we have
	
	\begin{enumerate}
		\item  $a_{l}^{(b,c)\pi} = a_{(b,c)\pi} ^{r}$.
		\item $ a_{r}^{(b,c)\pi} = a_{(b,c)\pi} ^{l}$.
		\item $a^{\|(b,c)} =ba(caba+1-(aba)^{c\pi})^{-1}c=(baca+1-(aca)^{c\pi})^{-1}bac$.
		\item $a^{\|(c,b)}=(caba+1-(aba)^{c\pi})^{-1}cab=ca(baca+1-(aca)^{b\pi})^{-1}b$.

	\end{enumerate}
\end{theorem}

\begin{proof}
	
	To prove that, we should write at first the formula of $a^{\|(b,c)}$ and $a^{\|(c,b)}$.
	\smallskip
	
	We have $a \in R^{\|(b,c)}\cap R^{\|(c,b)}$. And by \cite[Theorem 1]{symetric} we get
	\begin{eqnarray}
		a^{\|(b,c)}=ba(aba)^{\|c}=(aca)^{\|b}ac ,
		\label{equa1} \end{eqnarray}
	and
	\begin{eqnarray}
		a^{\|(c,b)}=(aba)^{\|c}ab=ca(aca)^{\|b}.
		\label{equa2}
	\end{eqnarray}
	
	So we obtain
	
	$$
	\begin{aligned}
		a_{l}^{(b,c)\pi}&= a^{\|(b,c)}a&=(aca)^{\|b}aca \, ;\\
		a_{r}^{(b,c)\pi}&=aa^{\|(b,c)}&=aba(aba)^{\|c} \, ;\\
		a_{(b,c)\pi} ^{r}&=aa^{\|(c,b)}&=aca(aca)^{\|b} \, ; \\
		a_{(b,c)\pi} ^{l}&= a^{\|(c,b)}a&=(aba)^{\|c}aba.
	\end{aligned}
	$$
	
	As $aba\in comm(c)$ and $aca\in comm(b)$, then
	$(aba)^{\|c}aba=aba(aba)^{\|c}$ and $aca(aca)^{\|b}=(aca)^{\|b}aca$ by \cite[Theorem 10]{mary} and  it follows that 
	$$	a_{l}^{(b,c)\pi}= a_{(b,c)\pi} ^{r} \quad  \text{ and} \quad  a_{r}^{(b,c)\pi} = a_{(b,c)\pi} ^{l}.$$
	Using  \cite[Theorem 2.8]{polarity-along}, we obtain that
	$$
	(aba)^{\|c}=(caba+1-(aba)^{c\pi})^{-1}c.$$
	$$(aca)^{\|b}=	(baca+1-(aca)^{b\pi})^{-1}b.	
	$$	
	Substitute in  (\ref{equa1}) and (\ref{equa2}), we obtain the result of $(3)$ and $(4)$.
\end{proof}
	
	\begin{remark}
		
\begin{enumerate}
	\item 
 We can also write $a^{\|(b,c)}$ and $a^{\|(c,b)}$ by using the result of \cite[Proposition 5]{symetric} as follow:
	$$a^{\|(b,c)}=bac(abac)^{\sharp}=ba(caba)^{\sharp}c=b(acab)^{\sharp}ac=(baca)^{\sharp}bac,$$
	$$ a^{\|(c,b)}=cab(acab)^{\sharp}=ca(baca)^{\sharp}b=c(abac)^{\sharp}ab=(caba)^{\sharp}cab.$$
	
	\item  If $a$ is only (c,b) polar, this does not allow us to have the equalities in the previous  theorem, because we may not have the right and the left $(b,c)$-spectral idempotents of $a$ since $a$ may not be $(b,c)$-polar, as we showed in Example  \ref{exmple3.5}.
\end{enumerate}	
\end{remark}

An {\it involution} * is a bijection $x\mapsto x^{*}$ on $R$, which satisfies the following conditions  for all $a, b \in R$: \smallskip

 $ i)\; \;  (a^{*})^{*}=a$ ; \smallskip
 
  $ ii) \; \; (ab)^{*}=b^{*}a^{*}$ ; \smallskip
  
    $ iii) \; \;  (a+b)^{*}=a^{*}+b^{*}$.\smallskip 
    
     We say that $R$ is a *-ring if there is an involution on $R$.
	
	\begin{proposition}
		Let $R$ be a *-ring,  and let $ a,b ,c \in R$. Then  $a$ is $(b,c)$-polar if and only if $a^{*}$ is dually $(b^{*},c^{*})$-polar.		

	In this case  we have
	$$
	\begin{aligned}
	   	(a_{l}^{(b,c)\pi})^{*}&=(a^{*})_{(b^{*},c^{*})\pi} ^{r}.\\
		( a_{r}^{(b,c)\pi} )^{*}& =(a^{*})_{(b^{*},c^{*})\pi} ^{l}.	
	\end{aligned}
	$$
	\end{proposition}

	\begin{proof}
	
	We have $a$ is $(b,c)$-polar if and only if $a$ is $(b,c)$-invertible by theorem \ref{polarity}. Suppose that $y=a^{\|(b,c)}$ so $y\in bR\cap Rc$ ;  $yab=b$ and $cay=c $.
	By involution we get $y^{*}\in c^{*}R\cap Rb^{*} $ ; $b^{*}a^{*}y^{*}=b^{*}$ and $y^{*}a^{*}c^{*}=c^{*}$ which means that $a^{*}$ is $(c^{*},b^{*})$-invertible with inverse $y^{*}$ i.e ; $y^{*}=(a^{\|(b,c)})^{*}=a^{*\|(c^{*},b^{*})}$.
	And $a^{*}$ is $(c^{*},b^{*})$-invertible  equivalent to $a^{*}$ is dually $(b^{*},c^{*})$-polar. And we have
	$$
	\begin{aligned}
		(a_{l}^{(b,c)\pi})^{*}&=(a^{\|(b,c)}a)^{*}=a^{*}(a^{\|(b,c)})^{*}=a^{*}a^{*\|(c^{*},b^{*})}=   (a^{*})_{(b^{*},c^{*})\pi} ^{r},
		\\
		( a_{r}^{(b,c)\pi} )^{*}&=	(aa^{\|(b,c)})^{*}=(a^{\|(b,c)})^{*}a^{*}=a^{*\|(c^{*},b^{*})}a^{*}=(a^{*})_{(b^{*},c^{*})\pi} ^{l}.				
	\end{aligned}
$$	
\end{proof}

\begin{proposition}
	Let $a,b,c \in R$ such that $a$ is $(b,c)$-polar and let $k\geq 1$. If $a\in comm(b)\cap comm(c)$ and $ba=ca$,  then we have \begin{enumerate}
		\item $a$ is polar along $b$ and $a^{b\pi}=a_{l}^{(b,c)\pi}$.
		\item $a$ is dually polar along $c$ and $a_{c\pi}=a_{r}^{(b,c)\pi}$.
		\item $a^{k}$ is $(b^{k},c^{k})$-polar, with $ (a^{k})_{l}^{(b^{k},c^{k})\pi}=a_{l}^{(b,c)\pi} $ and $(a^{k})_{r}^{(b^{k},c^{k})\pi}=a_{r}^{(b,c)\pi}$.
		
	\end{enumerate}

\end{proposition}

\begin{proof}

Since $a$ is $(b,c)$-polar then there exist $p=a_{l}^{(b,c)\pi}$ and $q=a_{r}^{(b,c)\pi}$ such that
 $$p=p^{2}\in bRca \; ;  \; q=q^{2}\in abRc \;  ;  \; pb=b \; ; \; cq=c \; ; \; cap=ca \;  \text{  and } \; qab=ab.$$

    $(1)$ and $(2)$:  As $ba=ca=ac=ab$ then  $bap=ba=pba$ which means that $p\in comm(ba)$ and $acq=ac=qac$ which means that $q\in comm(ac)$. Also $p\in bRba \subseteq Rba $ and $
	q\in abRc=acRc\subseteq acR$. Then $p=tba$ for some $t\in R$ and $q=acx$ for some $x\in R$. Then
	$$\begin{array}{ccl}
	(ptp+1-p)(ba+1-p)&=&ptpba+ptp(1-p)+(1-p)ba+1-p\\
	&=&ptba+0+0+1-p=p^{2}+1-p\\
	&=&1.
\end{array}$$
	
	Hence $(ba+1-p)\in R_{l}^{-1}$.  \smallskip
	And
$$\begin{array}{ccl}
			 (ac+1-q)(qxq+1-q)&=&acqxq+ac(1-q)+(1-q)qxq+1-q\\ &=&acxq+0+0+1-q\\
			 &=&q^{2}+1-q\\
			 &=&1.
\end{array}$$
	Hence $ac+1-q \in R_{r}^{-1}$.  
\smallskip

	By Jacobson's lemma  and the expression of $p=a^{\|(b,c)}a$ and $q=aa^{\|(b,c)}$, we have $$ 1+ba-p \in R_{l}^{-1} \Longleftrightarrow ac+1-q \in R_{l}^{-1} , $$
	and $$ 1+ ac-q \in R_{r}^{-1} \Longleftrightarrow ba+1-p \in R_{r}^{-1}.$$
	Thus we obtain $ba+1-p \in R^{-1}$ and $ac+1-q \in R^{-1}$. Consequently $a$ is polar along $b$ with $a^{b\pi}=p=a_{l}^{(b,c)\pi}$ and $a$ is dually polar along $c$ with $a_{c\pi}=q=a_{r}^{(b,c)\pi}$. \hfill
	
	$(3)$: Since $p$ and $q$ are idempotents and $a\in comm(b,c)$ we have $$c^{k}a^{k}p=(ca)^{k}p^{k}=(ca)^{k}=c^{k}a^{k},$$ and
	$$qa^{k}b^{k}=q^{k}(ab)^{k}=(ab)^{k}=a^{k}b^{k}.$$
	In other side, we have $pb^{k}=pbb^{k-1}=bb^{k-1}=b^{k}$ and $c^{k}q=c^{k-1}cq=c^{k-1}c=c^{k}$.
	
	Using $(1)$ and $(2)$ we have $$ \begin{aligned}	
	 ba+1-p \in R^{-1} & \Longrightarrow (ba+1-p)^{k}\in R^{-1} \\
	 &\Longleftrightarrow (ba)^{k}+1-p \in R^{-1} \Longleftrightarrow p\in (ba)^{k}R(ba)^{k} \; \text{ by }  \; \mbox{\cite[Theorem 2.4]{polarity-along}} \\
	 &\Longleftrightarrow p\in b^{k}a^{k}R(ca)^{k} \subseteq b^{k}Rc^{k}a^{k}.	
\end{aligned} $$	
	
Similarly $$      \begin{aligned}
 ac+1-q\in R^{-1} &\Longrightarrow (1+ac-q)^{k} \in R^{-1} \\
&\Longleftrightarrow (ac)^{k}+1-q \in R^{-1} \Longleftrightarrow  q\in (ac)^{k}R(ac)^{k}\\
&\Longleftrightarrow q\in (ab)^{k}Ra^{k}c^{k} \subseteq a^{k}b^{k}Rc^{k}
\end{aligned}.$$
Finally $a^{k}$ is $(b^{k},c^{k})$ polar with $(a^{k})_{l}^{(b^{k},c^{k})\pi}=p=a_{l}^{(b,c)\pi}$ and $(a^{k})_{r}^{(b^{k},c^{k})\pi}=q=a_{r}^{(b,c)\pi}$.

\end{proof}

\begin{theorem}
	Let $a, b, c$ and $ d \in R$ such that $a$ is $(b,c)$-polar. Then the following conditions are equivalent : \\
	\begin{enumerate}
		\item $d$ is $(b,c)$-polar such that $a_{l}^{(b,c)\pi}=d_{l}^{(b,c)\pi}$ and $a_{r}^{(b,c)\pi}=d_{r}^{(b,c)\pi}$.
		\\
	\item 	
$\left\{
\begin{array}{ll}
	cda_{l}^{(b,c)\pi}=cd \: ; \: a_{l}^{(b,c)\pi} \in bRcd  \mbox{  and  } a_{l}^{(b,c)\pi}b=b \\
	a_{r}^{(b,c)\pi}db=db  \: ; \:  a_{r}^{(b,c)\pi}\in dbRc \mbox{  and  }
	ca_{r}^{(b,c)\pi}=c
\end{array}
\right.$		
\item 	

$\left\{
\begin{array}{ll}
	cda_{l}^{(b,c)\pi}=cd \: ; \: a_{l}^{(b,c)\pi} \in bRcd\cap bRca  \mbox{  and  } a_{l}^{(b,c)\pi}b=b \\
	a_{r}^{(b,c)\pi}db=db  \: ; \:  a_{r}^{(b,c)\pi}\in abRc \cap dbRc \mbox{  and  }
	ca_{r}^{(b,c)\pi}=c
\end{array}
\right.$ 
		\item  $d$ is $(b,c)$-polar, $cda_{l}^{(b,c)\pi}=cd$ and $a_{r}^{(b,c)\pi}db=db$.		
	\end{enumerate}	
\end{theorem}
\begin{proof}

$(i) \Rightarrow (ii)$: is clear. \smallskip
\\
$(ii) \Rightarrow (i)$: by Definition \ref{def1} and Theorem \ref{thm.1} \smallskip
\\
$(ii) \Rightarrow (iii)$: we  have $ a_{l}^{(b,c)\pi} \in bRca$ and $a_{r}^{(b,c)\pi}\in abRc$, also we have by hypothesis $ a_{l}^{(b,c)\pi}\in bRcd$ and $a_{r}^{(b,c)\pi}\in dbRc$, so $a_{l}^{(b,c)\pi}\in bRcd \cap bRca$ and $a_{r}^{(b,c)\pi} \in abRc \cap dbRc$. \smallskip
\\
$(iii) \Rightarrow (ii)$:  it is  obvious \smallskip
\\
$(i) \Rightarrow (iv)$: is clear. \smallskip
\\
$(iv) \Rightarrow (i)$: we prove that $a_{l}^{(b,c)\pi}=d_{l}^{(b,c)\pi}$ and $a_{r}^{(b,c)\pi}=d_{r}^{(b,c)\pi}$. At first we have $d_{l}^{(b,c)\pi} \in bRcd$ so $d_{l}^{(b,c)\pi}=bxcd$ for some $x \in R$; and $a_{l}^{(b,c)\pi} \in bRca$ so $a_{l}^{(b,c)\pi}=btca$ for some $t \in R$. Moreover  $cdd_{l}^{(b,c)\pi}=cd=cda_{l}^{(b,c)\pi}$ and $a_{l}^{(b,c)\pi}b=b=d_{l}^{(b,c)\pi}b$, thus
 $$
d_{l}^{(b,c)\pi}-d_{l}^{(b,c)\pi}a_{l}^{(b,c)\pi}=bxcd-bxcda_{l}^{(b,c)\pi}=bxcd-bxcd=0.$$
 Hence
  $$d_{l}^{(b,c)\pi}=d_{l}^{(b,c)\pi}a_{l}^{(b,c)\pi}.$$
And $$a_{l}^{(b,c)\pi}-d_{l}^{(b,c)\pi}a_{l}^{(b,c)\pi}=btca-d_{l}^{(b,c)\pi}btca=btca-btca=0.$$
So $$a_{l}^{(b,c)\pi}=d_{l}^{(b,c)\pi}a_{l}^{(b,c)\pi}.$$
Consequently, we get $a_{l}^{(b,c)\pi}=d_{l}^{(b,c)\pi}$.
\smallskip

 Similarly  we show that  $d_{r}^{(b,c)\pi}=a_{r}^{(b,c)\pi}$.

\end{proof}

\section{The $(B,C)$-polarity for bounded linear operators}

Let $X$ be a complex Banach space and $\mathcal{B}(X)$ denotes the algebra of all bounded linear operators. Let $A\in \mathcal{B}(X)$, we will denote by $\mathcal{N}(A)=\{x\in X \;  :  \; Ax=0 \; \}$ the nullspace of $A$ and by $\mathcal{R}(A)=\{ Ax \; :  \; x\in X \; \}$ the range of $A$ and we write $I\in \mathcal{B}(X)$ for the identity operator.  \par

Let $A$, $B$, $C$ $\in \mathcal{B}(X)$, then $A$ is $(B,C)$-polar if there exist two projections  $P$ and $Q$ $\in \mathcal{B}(X)$ such that 

\begin{enumerate}
	\item $P\in B\mathcal{B}(X) CA$;
	\item  $Q\in  AB\mathcal{B}(X)C$; 
	\item $PB=B$, $CQ=C$;
	\item  $CAP=CA$, $QAB=AB$.	
\end{enumerate}

A closed subspace $M$ of $X$ is a complemented subspace of $X$ if there exits a closed subspace $N$ of $X$ such that $X=M\oplus N$.
\smallskip

Recall that an  operator $A\in \mathcal{B}(X)$ is regular if and only if $\mathcal{R}(A)$ is closed and a complemented subspace of $X$ and $\mathcal{N}(A)$ is a complemented subspace of $X$ (see \cite[Proposition 13.1]{Muller}).

\begin{theorem}\label{lagh1}
	Let $A$, $B$, $C$ $\in  \mathcal{B}(X)$ such that $B$, $C$ and $CAB$ are regular. Then the following assertions are equivalent :
	\begin{enumerate}
		\item 		
		$A$ is $(B,C)$-invertible ;
		\item 
		 $A$ is $(B,C)$-polar ;
		\item 
		There exist projections $P$, $Q$ $ \in  \mathcal{B}(X)$ such that
		
		 	\indent $ i) \;  \mathcal{R}(P)=\mathcal{R}(B)$ ; \smallskip
		 	
		 		\indent	$  ii) \; \mathcal{N}(Q)= \mathcal{N}(C)$ ; \smallskip
		 		
		 	\indent $ iii) \; \mathcal{R}(Q)=\mathcal{R}(AB)$ ; \smallskip
		 	
			\indent   $ iv)  \; \mathcal{N}(P)= \mathcal{N}(CA)$. 	
	\end{enumerate}	
\end{theorem}

\begin{proof}

 (1) $\Leftrightarrow (2) $ : By Theorem \ref{polarity}. \\
$(2) \Rightarrow (3) $ : \\
	$ i)$ From $P \in B\mathcal{B}(X) CA$ we have $$\mathcal{R}(P) \subseteq \mathcal{R}(B).$$
Also from $PB=B$, we see $\mathcal{R}(B)\subseteq \mathcal{R}(P)$. Hence $$\mathcal{R}(P) =\mathcal{R}(B).$$
 $ii) \;  CQ=C \Rightarrow \mathcal{N}(Q) \subseteq \mathcal{N}(C)$. \\
Now let $x\in X$ such that $Cx=0$. Since $Q\in AB\mathcal{B}(X)C$, then $Qx=0$ and so $\mathcal{N}(C)\subseteq \mathcal{N}(Q)$, thus $$\mathcal{N}(Q)=\mathcal{N}(C).$$
 $ iii) \; Q \in AB\mathcal{B}(X)C \Rightarrow \mathcal{R}(Q) \subseteq \mathcal{R}(AB)$. 
Also $QAB=AB \Rightarrow \mathcal{R}(AB) \subseteq \mathcal{R}(Q).$ 
$$\Rightarrow \mathcal{R}(Q)=\mathcal{R}(AB).$$
 $ iv) \; CAP=CA \Rightarrow \mathcal{N}(P)\subseteq \mathcal{N}(CA)$  
and $P\in B\mathcal{B}(X)CA \Rightarrow \mathcal{N}(CA) \subseteq \mathcal{N}(P)$ 
$$\Rightarrow \mathcal{N}(P)=\mathcal{N}(CA).$$

$ (3) \Rightarrow (1) $ : To show that $A$ is $(B,C)$-invertible, it suffices to prove that $\mathcal{N}(B)=\mathcal{N}(CAB)$ and $\mathcal{R}(C)=\mathcal{R}(CAB)$ by  virtue of \cite[Theorem 4.1]{deng}.
\smallskip

At first we can see that  $PB=B$ and $CQ=C.$ Indeed, let $x\in X.$ As $Bx\in \mathcal{R}(B) =\mathcal{R}(P)$ then $P(Bx)=Bx$ and so $PB=B$.\\
 We have $Cx=C(x-Q(x))+CQx)=0+CQx =CQx$ as $x-Qx \in \mathcal{N}(Q)=\mathcal{N}(C)$. Hence $C=CQ$. \smallskip

 Obviously we have $\mathcal{R}(CAB) \subseteq \mathcal{R}(C)$. Let $y\in \mathcal{R}(C)$ then there exists some  $ x \in X $ such that $y=Cx$. As $CQ=C$ we get $y=CQx$ and we can  write $y=Cz$ for some $z=Qx \in \mathcal{R}(Q) =\mathcal{R}(AB)$, so $z=ABt$ for some $t\in X$. Thus we obtain $y=Cz=CABt $ which implies that $y\in \mathcal{R}(CAB)$. Hence $\mathcal{R}(C) \subseteq \mathcal{R}(CAB)$ and consequently $$\mathcal{R}(C)=\mathcal{R}(CAB).$$ 

In the other side, we have obviously $\mathcal{N}(B) \subseteq \mathcal{N}(CAB)$. Suppose that $s\in \mathcal{N}(CAB)$, then $CABs=0$, set $z=Bs$, so we get $CAz=0$ which means that $z\in \mathcal{N}(CA)=\mathcal{N}(P)$, which implies $Pz=0 =PBs$. Since $PB=B$, we obtain $Bs=0$ which gives $s \in \mathcal{N}(B)$. Hence $\mathcal{N}(CAB) \subseteq \mathcal{N}(B)$, and we conclude that $$\mathcal{N}(B)=\mathcal{N}(CAB).$$ Therefore $A$ is $(B,C)$-invertible.
\end{proof}

\begin{remark}
Assume that $A$ is $(B,C)$-invertible. Then with respect the decomposition $X=\mathcal{R}(P)\oplus\mathcal{N}(P)$ and $X=\mathcal{R}(Q)\oplus\mathcal{N}(Q)$  we have the following  matrix representation of $A$:

	$$\begin{array}{rcl}
		A=\left[
		\begin{array}{cc}
			QA&QA\\
			(I-Q)A&(I-Q)A\\
		\end{array}
		\right] : \;  \mathcal{R}(P)\oplus\mathcal{N}(P)  &\longrightarrow& \mathcal{R}(Q)\oplus \mathcal{N}(Q) \\
		x_{1}+x_{2} &\longmapsto & A(x_{1}+x_{2})
	\end{array} $$ 
Indeed, we respect to the Priece decomposition we  have 
$$
	A=\left[
	\begin{array}{cc}
		QAP&QA(I-P)\\
		(I-Q)AP&(I-Q)A(I-P)\\
	\end{array}
	\right].
$$
Then for $x_{1}\in \mathcal{R}(P)$, $x_{2}\in \mathcal{N}(P)$ we have 
\begin{itemize}
	\item 
	$\begin{array}{rcl}
		QAP: \; \mathcal{R}(P) &\to&\mathcal{R}(Q) \\
		x_{1} &\mapsto &QAPx_{1}=QAx_{1}  \; ;
	\end{array} $

	\item 
	$\begin{array}{rcl}
		(I-Q)AP: \; \mathcal{R}(P) &\to&\mathcal{N}(Q) \\
		x_{1} &\mapsto &(I-Q)APx_{1}=(I-Q)Ax_{1}  \; ;
	\end{array} $

	\item 
	$\begin{array}{rcl}
		QA(I-P): \; \mathcal{N}(P) &\to&\mathcal{R}(Q) \\
		x_{2} &\mapsto &QAx_{2}  \; ;
	\end{array} $

	\item 
	$\begin{array}{rcl}
		(I-Q)A(I-P): \; \mathcal{N}(P) &\to&\mathcal{N}(Q) \\
		x_{2} &\mapsto &(I-Q)Ax_{2}.
	\end{array} $ 
\end{itemize}
\end{remark}

\begin{corollary}
	Let $A$ and $B \; \in \mathcal{B}(X)$ such that $B$ is regular. Then the following assertions are equivalent : 
	\begin{enumerate}
		\item $A$ is invertible along $B$ ;
		\item $A$ is polar along $B$ ;
		\item There exists  a projection $P \in \mathcal{B}(X)$ such that \smallskip
		
		\indent $i) \; \mathcal{N}(P)=\mathcal{N}(BA)=\mathcal{N}(B)$ ; \smallskip 
		
		\indent $ii) \; \mathcal{R}(P)=\mathcal{R}(AB)=\mathcal{R}(B)$.
		\item $\mathcal{R}(B)$ is closed and a complemented subspace of $X$, $\mathcal{R}(AB)$ is closed such that $X=\mathcal{R}(AB)\oplus \mathcal{N}(B)$ and the operator $A|\mathcal{R}(B)\,:\,\mathcal{R}(B)\rightarrow \mathcal{R}(AB)$ is invertible.		
		
	\end{enumerate} 
	
\end{corollary}

\begin{proof}
	The equivalence between 1), 2) and 3) follows from Theorem \ref{lagh1}. The equivalence between 1) and 4) is \cite[Theorem 2]{oper}, however we can give another proof by showing that 2) or 3) is equivalent to 4). 
\smallskip

	Indeed, assume that $A$ is polar along $B$, then by (3) $\mathcal{R}(B)=\mathcal{R}(P)$ which is closed and complemented in $X$ since $P$ is a bounded projection. Also $\mathcal{R}(AB)=\mathcal{R}(P)$ is closed and $X=\mathcal{R}(P)\oplus \mathcal{N}(P)=\mathcal{R}(AB)\oplus \mathcal{N}(B)$.
\smallskip

	The operator $A|\mathcal{R}(B)$ is surjective by construction. So let $x\in \mathcal{R}(B)$ such that $ABx=0$, then $BABx=0$ and hence $Bx\in\mathcal{N}(BA)=\mathcal{N}(P)$ by (3). Thus $0=PBx=Bx$. Therefore, $A|\mathcal{R}(B)$ is injective.
\smallskip
	
	Conversely, assume that (4) holds. Let $P$ be the bounded projection onto $\mathcal{R}(AB)$. Let $x\in X$ with $x=x_1+x_2$ such that $x_1\in\mathcal{R}(AB)=\mathcal{R}(P)$ and $x_2\in \mathcal{N}(B)=\mathcal{N}(P).$ Then $$BPx=Bx_1=B(x_1+x_2)=Bx.$$
	So $BP=B$. Also
	$$ABPx=ABx=ABx_1=PABx_1=PABx.$$
	Hence $P\in comm(AB)$. 
\smallskip

	Now to deduce that $A$ is dually polar along $B$, it remains to show that $AB+I-P$ is invertible. We have $(AB+I-P)x=ABx_1+x_2$. Then if $(AB+I-P)x=0$, we get $ABx_1=0$ and $x_2=0$. Since the operator $A|\mathcal{R}(B)$ is invertible, we deduce that $x_1=0$. Thus $AB+I-P$ is injective. 
\smallskip

	Let $y=y_1+y_2$ such that $y_1\in\mathcal{R}(AB)$ and $y_2\in \mathcal{N}(B).$ Then $y_1=ABx=ABx_1$ for some $x\in X$. Set $z=x_1+y_2$, then $(AB+I-P)z=ABx_1+y_2=y$. Hence $AB+I-P$ is surjective. Therefore $A$ is dually polar along $B$ and so $A$ is polar along $B$.
\end{proof}

\noindent{\bf Date Availability Statement:} No data were produced for this paper.\\
{\bf Conflict of interest:} The authors declare that they have no conflict of interest.

\end{document}